\documentclass[11pt,reqno]{amsart}
\usepackage{geometry}                
\usepackage{graphicx}
\usepackage{amsmath}
\usepackage{amssymb}
\usepackage{epstopdf}
\usepackage{float}
\usepackage{enumerate}
\usepackage{esvect}
\usepackage[linesnumbered,ruled,vlined,norelsize]{algorithm2e}
\usepackage{chngcntr}

\makeatletter
\def\subsubsection{\@startsection{subsubsection}{3}%
  \z@{.5\linespacing\@plus.7\linespacing}{.1\linespacing}%
  {\normalfont\itshape}}
\makeatother

\newtheorem{theorem}{Theorem}
\counterwithin*{theorem}{section}

\theoremstyle{definition}

\newtheorem{proposition}{Proposition}
\counterwithin*{proposition}{section}

\newtheorem{problem}{Problem}
\counterwithin*{problem}{section}
\usepackage[all]{xy}
\usepackage{graphicx}
\theoremstyle{remark}

\title{Minimal curvature-constrained networks}
\author{D. Kirszenblat}

\author{K. Sirinanda}

\author{M. Brazil}

\author{P. Grossman}

\author{J. H. Rubinstein}

\author{D. Thomas}

\begin{document}

\begin{abstract}
This paper introduces an exact algorithm for the construction of a shortest curvature-constrained network interconnecting a given set of directed points in the plane and an iterative method for doing so in 3D space. Such a network will be referred to as a  minimum Dubins network, since its edges are Dubins paths (or slight variants thereof).  The problem of constructing a minimum Dubins network appears in the context of underground mining optimisation, where the aim is to construct a least-cost network of tunnels navigable by trucks with a minimum turning radius. The Dubins network problem is similar to the Steiner tree problem, except that the terminals are directed and there is a curvature constraint. We propose the minimum curvature-constrained Steiner point algorithm for determining the optimal location of the Steiner point in a 3-terminal network. We show that when two terminals are fixed and the third varied, the Steiner point traces out a lima\c{c}on.
\end{abstract}


\maketitle

\section{Introduction}\label{intro}
This paper is concerned with the problem of designing a shortest path network for vehicles. The paths in the network are subject to a curvature constraint that accounts for the minimum turning radius of the vehicles. Interest in such networks is motivated by their relevance for designing the access in underground mines \cite{Brazil_b}.

The Dubins network problem can be viewed as a novel combination of two problems that have been well-studied in the optimisation literature: the Steiner problem and the Dubins problem. What follows is a review of these two problems. Some basic geometric and topological concepts primarily drawn from \cite{Dubins} and \cite{Hwang} are presented here in order to prepare the reader for subsequent sections. 

First, consider Fermat's problem, which is a 3-terminal special case of the Steiner problem:

\begin{problem}[Fermat's problem]
Given three points $p_{1}, p_{2}, p_{3}$ in the plane $\mathbb{R}^2$, find the point $s$ which minimises the sum of the distances $||sp_1|| + ||sp_2|| + ||sp_3||$.
\end{problem}

The uniqueness of $s$, called the \emph{Steiner point}, is obtained from the convexity of the Euclidean norm. If one of the angles of $\triangle p_{1}p_{2}p_{3}$ is at least $120^\circ$, then $s$ is located at its vertex. Otherwise, $s$ lies in the interior of $\triangle p_{1}p_{2}p_{3}$, whose sides subtend angles of $120^\circ$ at $s$. Melzak \cite{Melzak} proposed a ruler and compass construction for finding $s$ in the latter case. Let $m$ be the third vertex of the equilateral triangle with $p_{1}$ and $p_{2}$ as its other two vertices, and whose interior lies outside that of $\triangle p_{1}p_{2}p_{3}$. Let $\Gamma$ be the circle through $p_{1}, p_{2}, m$. Then $s$ is the intersection of $\Gamma$ and the \emph{Simpson line} $mp_{3}$ as shown in Fig. \ref{FIG1}.

\begin{figure}	[h!]
	\centering
	\includegraphics[scale=0.9]{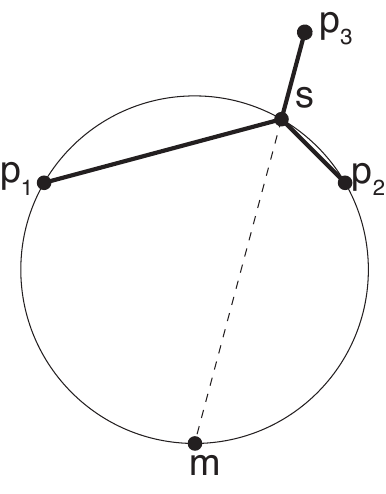}
	\caption{Melzak's construction}
	\label{FIG1}
\end{figure}

Based on the above 3-terminal algorithm that substitutes $m$ for $p_{1}$ and $p_{2}$, locates $s$ and adds the straight segments $p_{1}s$ and $p_{2}s$, the Melzak algorithm generalises to $n$ terminals (see \cite{Melzak} for details). The Melzak algorithm yields all the minimising networks of the Steiner problem:

\begin{problem}[The Steiner problem]
Given $n$ points $p_1, \ldots,p_n$ in the plane $\mathbb{R}^2$, construct a shortest network interconnecting these $n$ points.
\end{problem}

Some basic moves for constructing a shortest network are as follows. Consider a network $S$ interconnecting $n$ points $p_{1}, \ldots, p_n,$ called \emph{terminals}, in the plane $\mathbb{R}^2$. A \emph{Steiner point} is any vertex in $S$ other than a terminal. A Steiner point must be of degree at least three. To see this, a degree-1 Steiner point can be deleted along with its incident edge to shorten the network. Also, a degree-2 Steiner point and its two incident edges can be deleted and replaced by a single edge.

\emph{Splitting} a vertex $v_2$ is the operation of disconnecting two straight segments $v_1v_2, v_2v_3$ and connecting $v_1, v_2, v_3$ to a newly created Steiner point $s$. \emph{Shrinking} an edge $sv_2$ is the reverse operation which returns the original graph. Suppose that two straight segments $v_1v_2, v_2v_3$ meet at a vertex $v_2$ with an angle less than $120^\circ$. Then the network can be shortened by splitting at $v_2$ and solving Fermat's problem for $v_1, v_2, v_3$ in order to locate the new Steiner point $s$. Consequently, a Steiner point must be of degree at most three. It follows that a Steiner point must be of degree exactly three.

Gilbert and Pollak \cite{Gilbert2} called a local minimum network a \emph{Steiner tree}. The network is a local minimum in the sense that it cannot be shortened by small perturbation of the Steiner points, even when splitting is allowed. They called a global minimum network a \emph{Steiner minimal tree}. A simple counting argument shows that a Steiner tree with $n$ terminals has at most $n-2$ Steiner points \cite{Courant}. A Steiner tree with $n-2$ Steiner points is called a \emph{full Steiner tree}.

Now consider the Dubins problem in the plane. It is convenient to introduce some notation. A directed point is a pair $(p, \vv p)$, where $p$ is a point in the plane $\mathbb{R}^2$ and $\vv p$ is a tangent vector in the tangent space $T_p\mathbb{R}^2$ based at $p$. For simplicity, such a pair is denoted as $\mathbf p$. Given two directed points $\mathbf p_1, \mathbf p_2,$ a path connecting these two directed points is called \emph{admissible} \cite{Ayala} if:

\begin{enumerate}[(i)]
	\item It is continuously differentiable.
	\item It starts at $p_1$ and finishes at $p_2$ and has tangent vectors $\vv p_1$ and $\vv p_2$ at these two points respectively.
	\item It is piecewise twice differentiable. There is a finite number of points at which the curvature is not defined.
	\item Where the curvature exists it is less than or equal to a fixed positive number $\rho^{-1}$, where $\rho$ is the minimum radius of curvature.
\end{enumerate}

Without loss of generality, for the theory we will set $\rho = 1$ throughout the paper. The Dubins problem can be formulated as follows:

\begin{problem}[The Dubins problem]
Given an initial directed point $\mathbf p_1$ and a final directed point $\mathbf p_2$, construct a shortest admissible path connecting these two directed points.
\end{problem}

Dubins \cite{Dubins} proved that such a path, called a \emph{Dubins path}, necessarily exists and consists of not more than three components, each of which is either a straight segment, denoted by $S$, or an arc of a circle of radius $\rho$, denoted by $C$. Furthermore, such a path is necessarily a subpath of a path of type $CSC$ or of type $CCC$. A well-known result states that if two points $p_1$ and $p_2$ are distance at least $4\rho$ apart, then a shortest Dubins path between any two directed points $(p_1, \vv p_1)$ and $(p_2, \vv p_2)$ is of type $CSC$ (see, for example, \cite{Goaoc}).

From now on this paper adopts the convention that given two directed points $\mathbf p_1, \mathbf p_2$, a path connecting these two directed points is admissible if the conditions (i), (iii) and (iv) stated above are fulfilled and it has tangent vectors $\vv p_1$ and $-\vv p_2$ at $p_1$ and $p_2$, respectively, when traveling from point $p_1$ to point $p_2$. We do not specify a starting point, because none of the terminals in a network can be considered as initial or final when the number of terminals $n$ is greater than or equal to 3.

Given three directed points $\mathbf p_1, \mathbf p_2, \mathbf p_3$ in the plane, a network interconnecting these three directed points is called \emph{admissible} if:
\begin{enumerate}[(i)]
	\item It is continuously differentiable except at any Steiner point.
	\item It interconnects $p_1, p_2, p_3$ and has tangent vectors $\vv p_1, \vv p_2, \vv p_3$ 		at these three points respectively.
	\item It is piecewise twice differentiable. There is a finite number of points at which the curvature is not defined.
	\item Where the curvature exists it is less than or equal to a fixed positive number $\rho^{-1}$,
	where $\rho$ is the minimum radius of curvature.
\end{enumerate}

A 3-terminal version of the Dubins network problem can be formulated as follows:

\begin{problem}[The 3-terminal Dubins network problem]
Given three directed points $\mathbf p_1, \mathbf p_2, \mathbf p_3,$ construct a shortest admissible network interconnecting these three directed points.
\end{problem}

An example of such a network is shown in Fig. \ref{FIG2}. By analogy with Steiner trees, a local minimum network is called a \emph{Dubins network}, whereas a global minimum network is called a \emph{minimum Dubins network}.

\begin{figure}	[h!]
	\centering
	\includegraphics[scale=0.8]{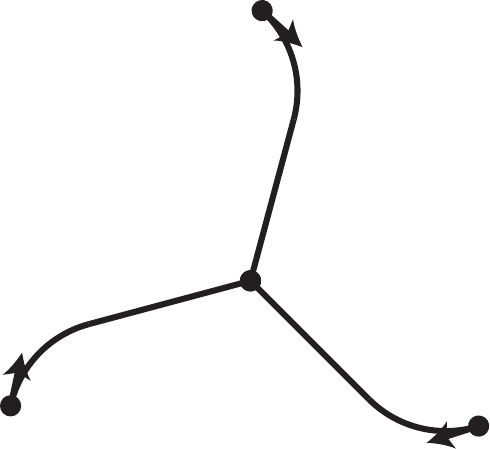}
	\caption{A 3-terminal Dubins network}
	\label{FIG2}
\end{figure}

\section{Dubins networks in the plane}\label{sec:2}
To simplify the analysis in this paper, we will focus on the full case of the 3-terminal problem. That is, we will focus on the case where the network has three edges of type $CS$ meeting at a Steiner point. On the other hand, the degenerate case is not particularly interesting and consists of a concatenation of two Dubins paths, or a slight modification thereof. Determining the conditions under which the degenerate case occurs is beyond the scope of this paper. We begin our analysis with a preliminary proposition. To this end, it will be convenient to adopt the variational approach from the work of Rubinstein and Thomas \cite{Rubinstein} on Steiner trees.

\begin{proposition}
	Suppose that $p_1 p_2 p_3$ is a curve of type $CS$ and that $p_1$ is held fixed while $p_3$ is varied. Then the first variation of length of $p_1 p_2 p_3$ is the negative of the scalar product between the direction of variation and the outward pointing unit vector $\hat{p_3 p_2}$.
	\label{PROP1}
\end{proposition}

	\begin{proof}
Let $p_1 p_2 p_3$ be a curve of type $CS$ where:
		
		\begin{itemize}
			\item $p_1 p_2$ is a circular arc of radius 1.
			\item $p_2 p_3$ is a straight segment.
		\end{itemize}
		
			\begin{figure}[h!]
		\centering
		\includegraphics[width=.45\linewidth]{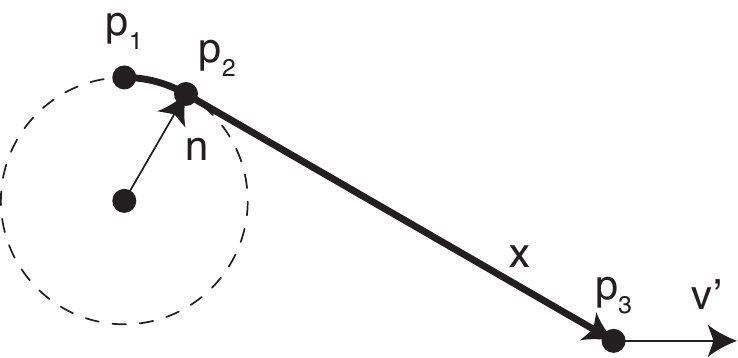}
		\caption{A curve of type $CS$}
		\label{FIG3}
	\end{figure}
		
		In Figure \ref{FIG3}, the point $p_3$ moves along any smooth curve with derivative at its initial position being the vector $\mathbf{v}'$. Let $L$ denote the length of $p_1 p_2 p_3$. The curve $p_1 p_2$ is an arc of a unit circle with centre the origin. The length of $p_1 p_2$ is the angle $\theta$ subtended by $p_1p_2$ at the origin. Let $\mathbf{\hat{n}}$ be the unit vector from the origin to the point $p_2$ and $\mathbf{x}$ the vector from $p_2$ to $p_3$, with length denoted by $x$ and unit vector in the direction of $\mathbf{x}$ denoted by $\mathbf{\hat x}$. Observe that the first variation of the vector $\mathbf{x}$ is equal to the first variation of its head, i.e. $\mathbf{v}'$, minus the first variation of its tail, i.e. $\mathbf{\hat{n}}'$. If the point $p_3$ is perturbed in the direction of $\mathbf{v}'$, then the first variation of length of $p_1 p_2 p_3$ is
		
		\begin{align*}
		L' &= \theta' + x'\\
		&= \mathbf{\hat{n}}' \cdot \mathbf{\hat{x}} + \mathbf{x}' \cdot \mathbf{\hat{x}} \\
		&= \mathbf{\hat{n}}' \cdot \mathbf{\hat{x}} + (\mathbf{v}' - \mathbf{\hat{n}}') \cdot \mathbf{\hat{x}} \\
		&= \mathbf{v}' \cdot \mathbf{\hat{x}}.
		\end{align*}
		
\end{proof}

The \emph{weighted Dubins network problem} is analogous to the Fermat-Weber problem (see, for example, \cite{Volz}). In the 3-terminal case, let $L_i$ denote the length of a CS-path from the $i$th terminal to the variable Steiner point. The weights $w_i$ are the costs per unit length of the three paths. We seek to minimise the weighted sum of path lengths $\sum_{i = 1}^3 w_i L_i$. We may appeal to a local argument in order to show that the angles at the junction are the same as those of a similarly weighted Steiner tree. To that end, consider the 3-terminal case of the Fermat-Weber problem, which is to minimise the weighted sum $\sum_{i=1}^3 w_i \|{\bf x}_i\|$. Here ${\bf x}_i$ denotes the vector from the $i$th terminal to the variable Steiner point. We use a variational argument. Noting that ${\bf x}_i' = {\bf v}'$ and ${\bf x}_i\cdot\hat{\bf x}_i' = 0$ (differentiate $\hat{\bf x}_i\cdot\hat{\bf x}_i$) throughout the perturbation in the direction of $\mathbf{v}'$, we have by the product rule,
\begin{align*}
L' &= \sum_{i=1}^3 w_i \|{\bf x}_i\|'\\
&= \sum_{i=1}^3 w_i ({\bf x}_i\cdot\hat{\bf x}_i)'\\
&= \sum_{i=1}^3 w_i ({\bf x}_i'\cdot\hat{\bf x}_i + {\bf x}_i\cdot\hat{\bf x}_i')\\
&= \sum_{i=1}^3 w_i {\bf v}'\cdot\hat{\bf x}_i\\
&= {\bf v}'\cdot\sum_{i=1}^3 w_i \hat{\bf x}_i.
\end{align*}
At equilibrium we must have $\sum_{i=1}^3 w_i \hat{\bf x}_i = 0$. Mechanically speaking, this means that the sum of the forces ${\bf f}_i = -w_i \hat{\bf x}_i$ acting at the Steiner point is zero. An application of the law of cosines gives the angles $\alpha_i$ at the junction in terms of the weights $w_i$.

\begin{proposition}
For any 3-terminal Dubins network whose paths are of type $CS$, the three Dubins paths meet at angles $\alpha_1, \alpha_2, \alpha_3$ determined as in the case of the Fermat-Weber problem.
\end{proposition}

	\begin{figure}[h!]
		\centering
		\includegraphics[width=.8\linewidth]{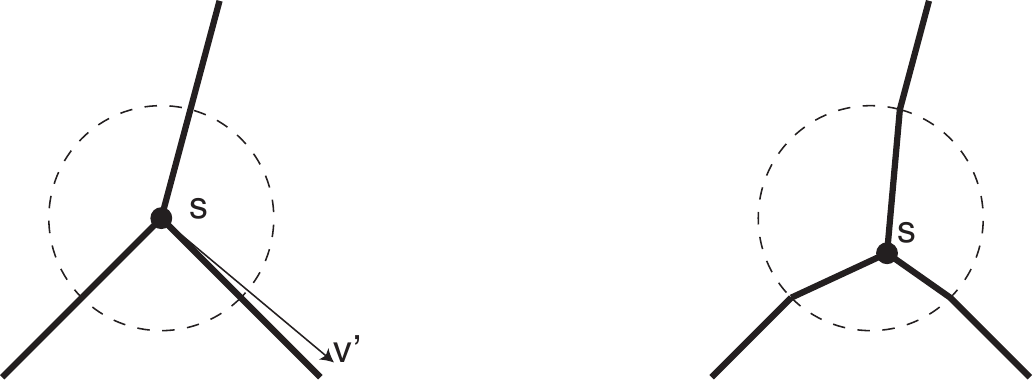}
		\caption{A perturbation in the direction of negative gradient will shorten the network.}
		\label{FIG4}
	\end{figure}

\begin{proof}
It suffices to work locally. Fix a small disk around the Steiner point (see Fig. \ref{FIG4} (left)). Fix the points of intersection of the Dubins paths with the boundary of the disk. Within the disk, if the three straight segments do not make angles $\alpha_1, \alpha_2, \alpha_3$ with each other, then move the Steiner point in the direction of negative gradient (see Fig. \ref{FIG4} (right)). (Recall that the gradient is the negative of the sum of the three weighted outward pointing unit vectors.) Finally, replace the newly created angles at the boundary of the disk by small arcs of circles of radius $\rho$. Note that two length-reducing moves have been made. This shows that the three straight segments meet at angles $\alpha_1, \alpha_2, \alpha_3$.
\end{proof}

We now show that when two terminals are fixed and the third varied, the Steiner point traces out a lima\c{c}on. The lima\c{c}on itself is not required by the algorithm. However, certain geometric constructions (i.e. lines and circles) found in the derivation of the equation of the lima\c{c}on will be used in the algorithm. Assume that all of the paths are of type $CS$. Choose one Dubins circle for each directed point. (Note that there are eight possible combinations, and that each combination can be considered independently.) Take any two of the Dubins circles. There are two possible Dubins topologies that can arise: an \emph{even} topology arises when the circular arcs are similarly oriented, and an \emph{odd} topology arises when the circular arcs are oppositely oriented. For example, the topology shown in Fig. \ref{FIG5} (right) is even, because both circular arcs, i.e. $C$ components incident to $p_1$ and $p_2$, are clockwise oriented. On the other hand, the topology shown in Fig. \ref{FIG5} (left) is odd.

\begin{figure}[h!]
	\centering
	\includegraphics[width=\linewidth]{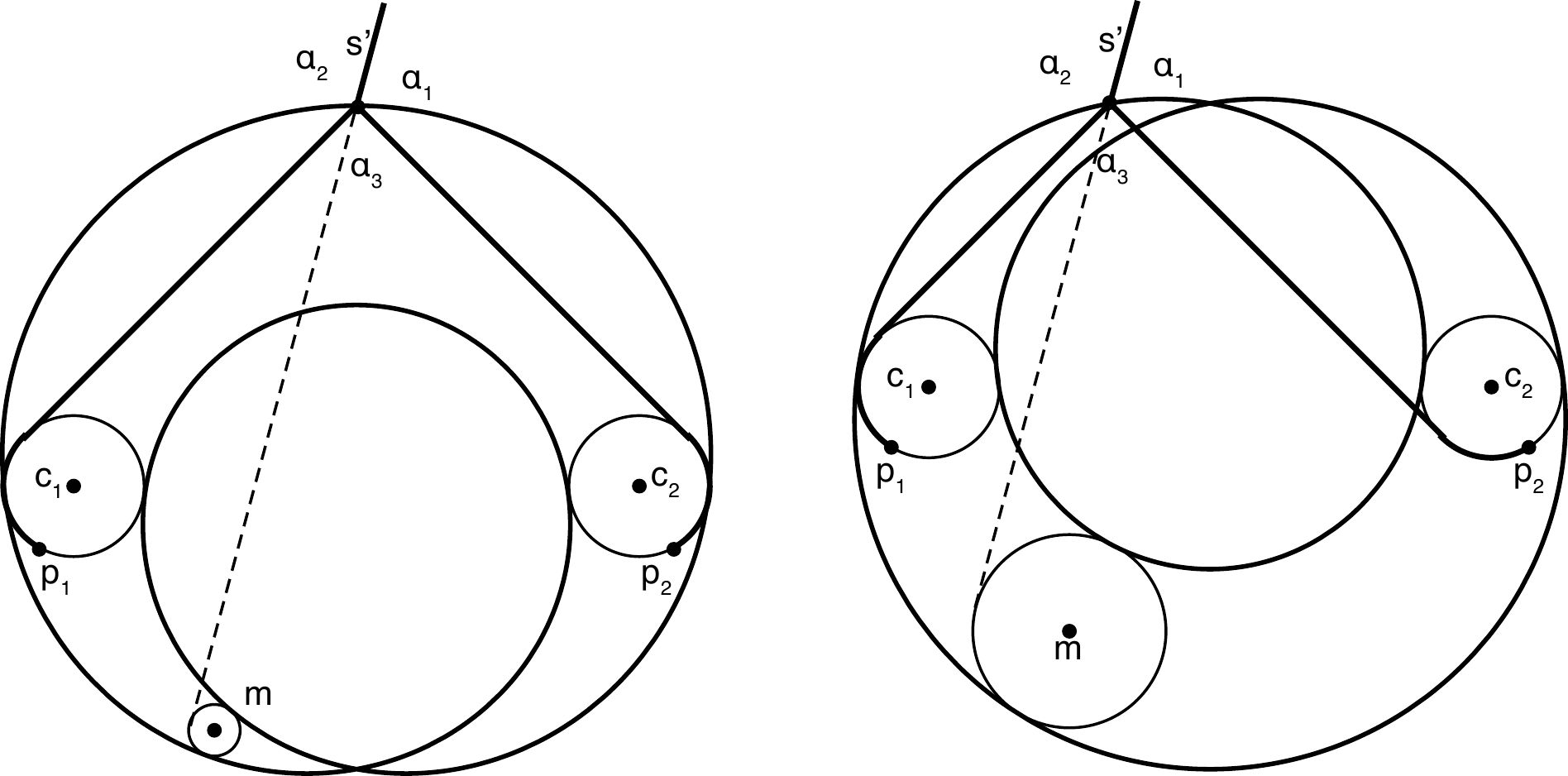}
	\caption{The loci of potential Steiner points for two different topologies}
	\label{FIG5}
\end{figure}

We first need to construct the following circles (see Figures \ref{FIG6} through \ref{FIG8}):
\begin{enumerate}[(i)]
\item The auxiliary circle through points $c_1$, $c_2$, and $s$.
\item The Melzak circle with centre $m$.
\item The circle centred at $s$ and incident to $s'$.
\end{enumerate}
Only the first and last circles will be used in the derivation of the equation of the lima\c{c}on, whereas all three will be used in the algorithm in section 3.

\begin{figure}[h!]
	\centering
	\includegraphics[width=\linewidth]{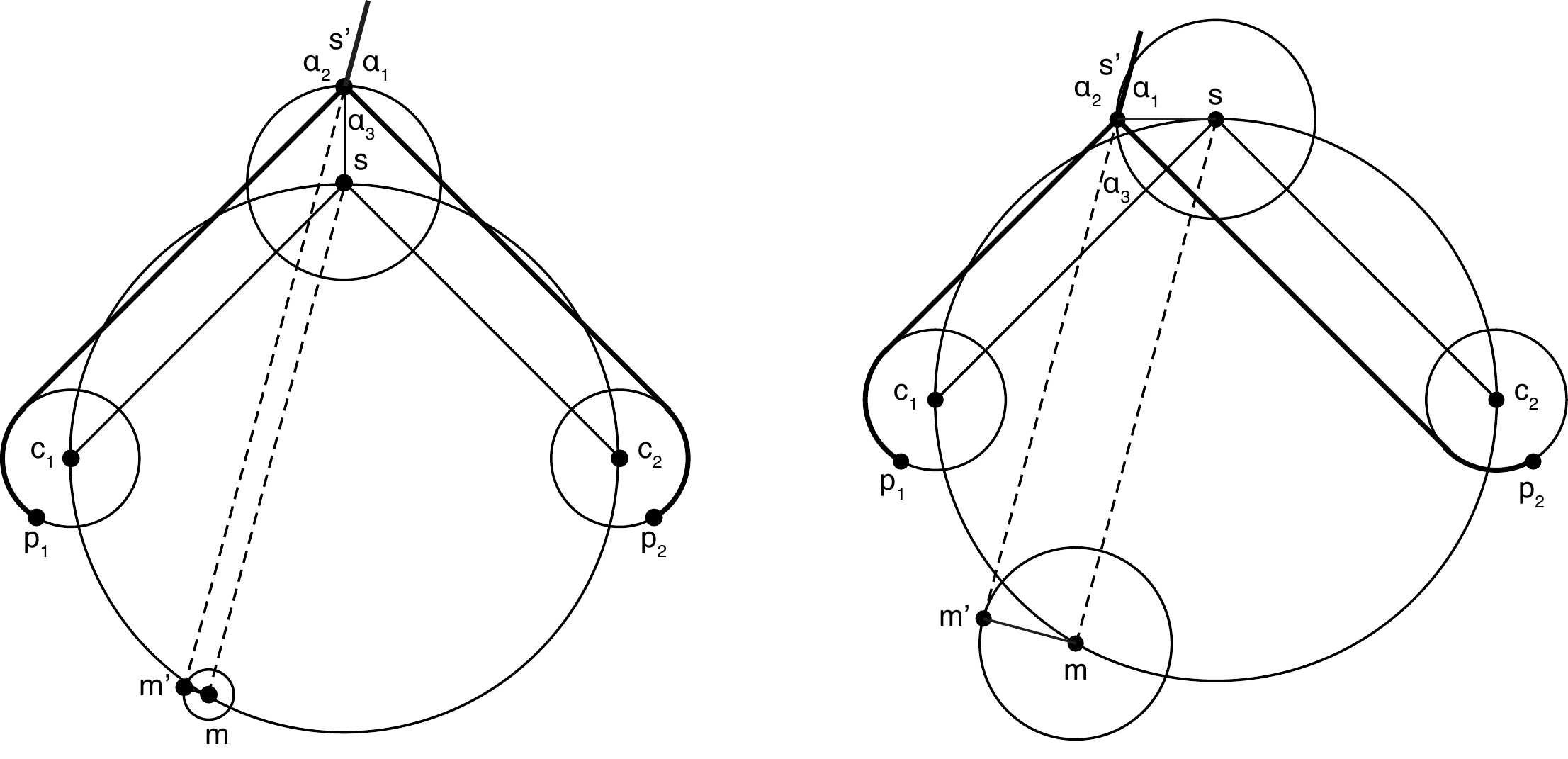}
	\caption{The Melzak-like construction for two different topologies}
	\label{FIG6}
\end{figure}

In order to determine the auxiliary circle, consider two segments $c_1s$ and $c_2s$ parallel to  the first and second paths, respectively, of the sought-after Dubins network. (Refer to Fig. \ref{FIG6}) Since the angle $c_1sc_2 = \alpha_3$, by the Inscribed Angle Theorem, $s$ lies on a circular arc of angle $2\pi - 2\alpha_3$ through $c_1$ and $c_2$. Now let $d$ denote the distance between the two Dubins circles with centres $c_1$ and $c_2$ and $r$ the radius of the auxiliary circle through points $c_1$, $c_2$, and $s$. Using basic trigonometry, we have $\sin(\pi - \alpha_3) = \frac{d}{2r}$, and hence $r = \frac{1}{2} d \csc\alpha_3$.

For either topology, locating the Melzak point $m$ follows the same procedure as in the Fermat-Weber problem \cite{Gilbert1}. Note that the exterior angles of the triangle $c_1c_2m$ are $\alpha_1, \alpha_2, \alpha_3$ (see Fig. \ref{FIG7}). Then this triangle is similar to that formed by the three force vectors ${\bf f}_i$. Hence the Melzak point $m$ is determined. Let $r_m$ and $r_s$ denote the radii of the Melzak circle and the circle whose radius is equal to the distance between between the points $s$ and $s'$, respectively. In order to determine the radius $r_m$ in the case of the odd topology, refer to Figure \ref{FIG8} (left). Observe that $r_m = r_s \sin(\pi - \alpha_1 - \frac{\alpha_3}{2})$ and $1 = r_s \sin(\frac{\alpha_3}{2})$. Eliminating $r_s$ and solving for $r_m$ we get: $r_m = \sin (\alpha_1) \cot \left(\frac{\alpha_3}{2}\right) + \cos (\alpha_1)$. In order to determine the radius $r_m$ in the case of the even topology, refer to Figure \ref{FIG8} (right). Observe that $r_m = r_s \cos(\pi - \alpha_1 - \frac{\alpha_3}{2})$ and $1 = r_s \sin(\frac{\pi}{2}-\frac{\alpha_3}{2})$. Eliminating $r_s$ and solving for $r_m$ we get: $r_m = \sin (\alpha_1) \tan \left(\frac{\alpha_3}{2}\right) - \cos (\alpha_1)$.

\begin{figure}[h!]
	\centering
	\includegraphics[width=0.6\linewidth]{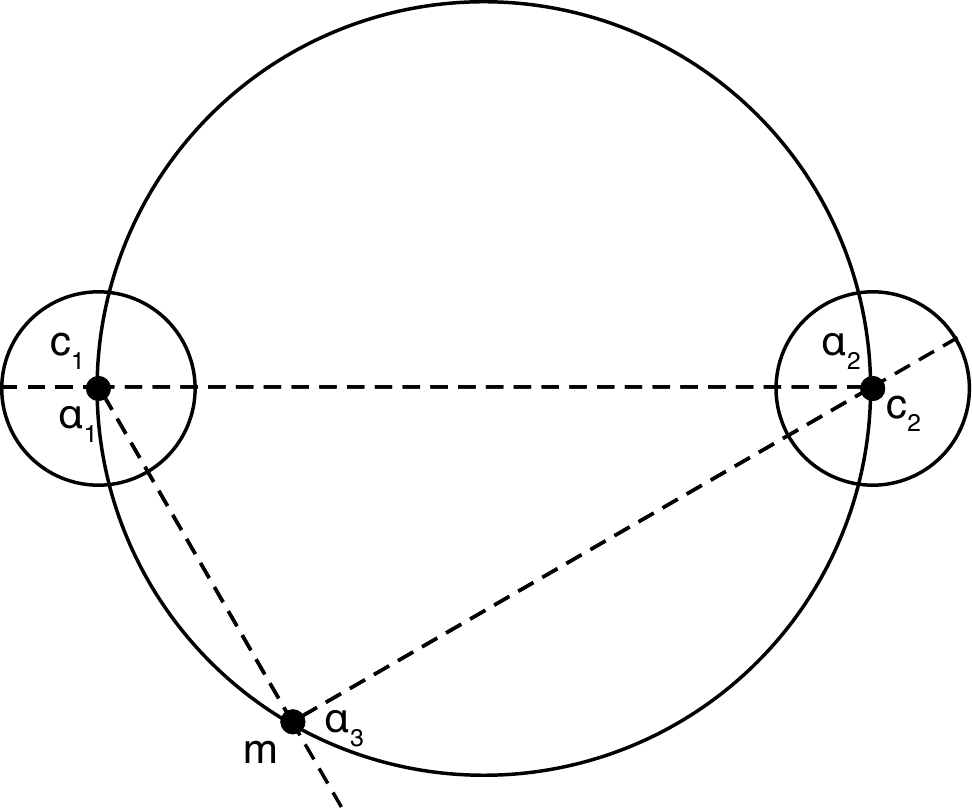}
	\caption{Locating the Melzak point}
	\label{FIG7}
\end{figure}

\begin{figure}[h!]
	\centering
	\includegraphics[width=\linewidth]{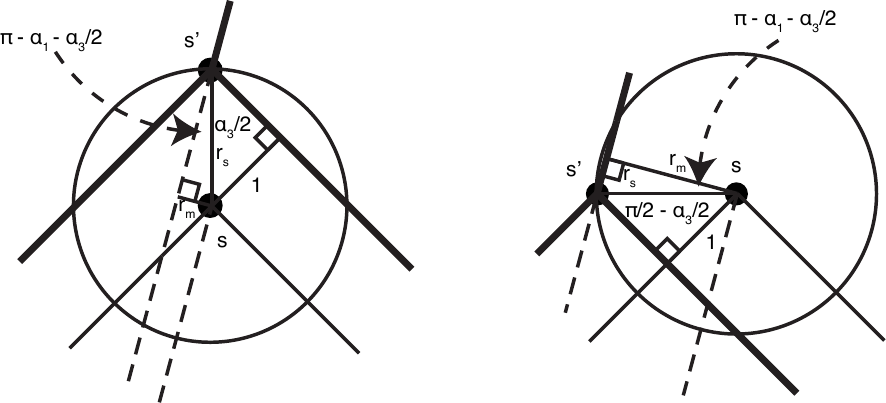}
	\caption{Determining the radii $r_m$ and $r_s$}
	\label{FIG8}
\end{figure}

Finally, we need to determine the distance between between the points $s$ and $s'$. Once again, refer to Figure \ref{FIG8}. We have $r_s = \csc \left(\frac{\alpha_3}{2}\right)$ and $r_s = \sec \left(\frac{\alpha_3}{2}\right)$ in the cases of the odd and even topologies, respectively.

\begin{theorem}
Consider a full Dubins network on three terminals. Suppose that two terminals are fixed and the third varied. Then the Steiner point traces out a lima\c{c}on.
\end{theorem}

\begin{proof}
Suppose we have determined the angles between the paths of a 3-terminal Dubins network. In particular, suppose the angle between the first and second paths is $\alpha$. The method for determining the lima\c{c}on (i.e., the locus of potential Steiner points) is as follows. Beginning with the odd topology, choose a polar coordinate system $(r, \theta)$ such that its pole is $p$ and the polar axis points toward $q$. Refer to Figure \ref{FIG9} (left). The point $q$ constitutes the centre of the auxiliary circle from the Melzak algorithm. The point $p$ is the point of self-intersection of the lima\c{c}on. Let $d$ denote the distance between the two Dubins circles. The three phasors $\mathbf v_1 = \vv{pq}$, $\mathbf v_2(\theta) = \vv{qs}$, and $\mathbf v_3(\theta) = \vv{ss'}$ are given by 
$\frac{1}{2} d \csc\alpha (1, 0)$, 
$\frac{1}{2} d \csc\alpha (\cos2\theta, \sin2\theta)$, and 
$\csc \left(\frac{\alpha }{2}\right) (\cos\theta, \sin\theta)$, 
respectively. This is easy to show using basic trigonometry. Then

\begin{align*}
\mathbf r(\theta) &= \mathbf v_1 + \mathbf v_2(\theta) + \mathbf v_3(\theta)\\
&= \frac{1}{2} d \csc\alpha (1, 0) + \frac{1}{2} d \csc\alpha (\cos2\theta, \sin2\theta) + \csc \left(\frac{\alpha }{2}\right) (\cos\theta, \sin\theta).
\end{align*}

After computing the norm of both sides and simplifying, we obtain for the equation of the lima\c{c}on 
\begin{equation*}
r(\theta) = \csc \left(\frac{\alpha }{2}\right)+d \csc (\alpha ) \cos (\theta ).
\end{equation*}

\begin{figure}
\centering
\includegraphics[width=\textwidth]{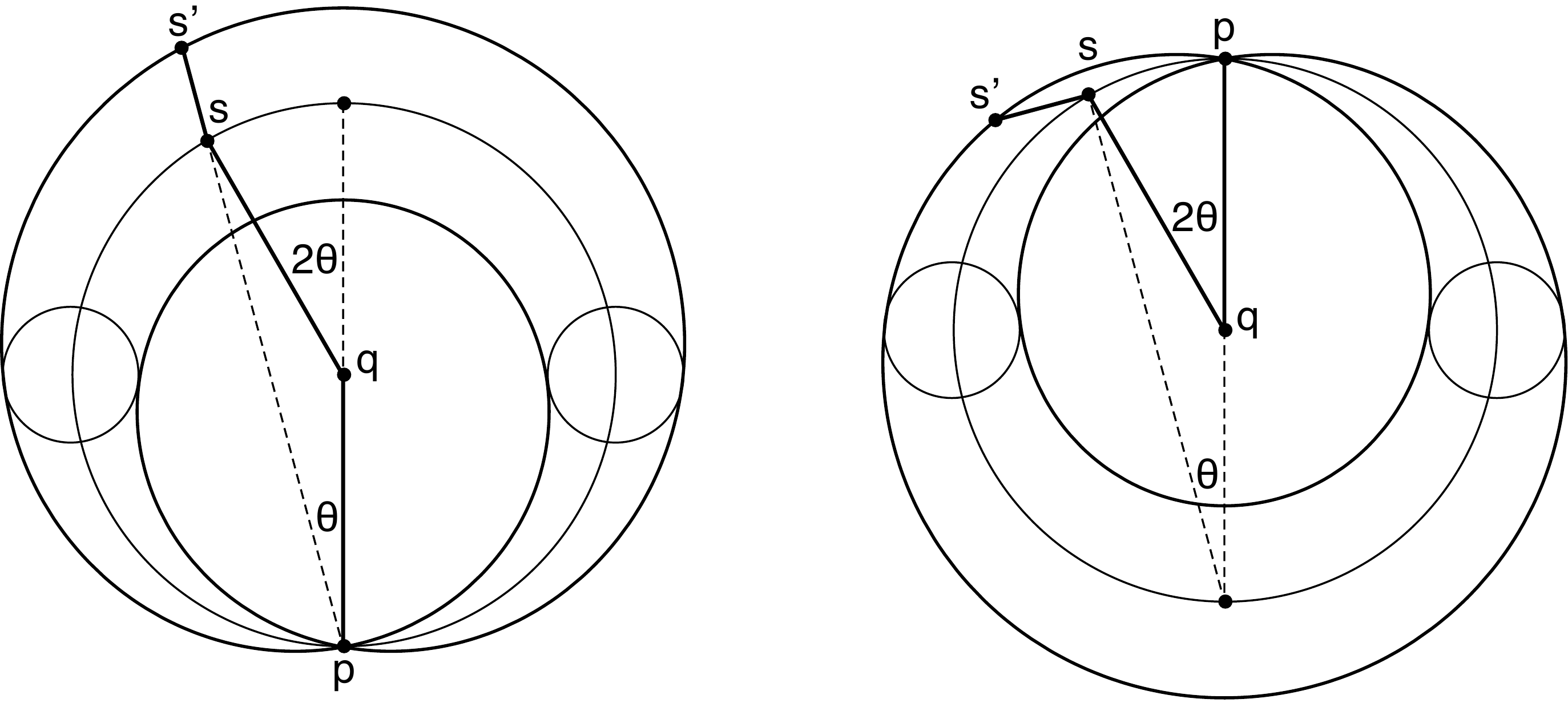}
\caption{The graphs of lima\c{c}ons for two different topologies}
\label{FIG9}
\end{figure}

For the odd topology, choose a polar coordinate system $(r, \theta)$ such that its pole is $p$ and the polar axis points toward $q$. The point $q$ constitutes the centre of the auxiliary circle from the Melzak algorithm. The point $p$ is the point of self-intersection of the lima\c{c}on. The three phasors $\mathbf v_1 = \vv{pq}$, $\mathbf v_2(\theta) = \vv{qs}$, and $\mathbf v_3(\theta) = \vv{ss'}$ are now given by $\frac{1}{2} d \csc\alpha (-1, 0)$, $\frac{1}{2} d \csc\alpha (\cos2\theta, \sin2\theta)$, and $\sec \left(\frac{\alpha }{2}\right) (-\sin\theta, \cos\theta)$, respectively. The resulting equation of the lima\c{c}on is
\begin{equation*}
r(\theta) = \sec \left(\frac{\alpha }{2}\right)-d \csc (\alpha ) \cos (\theta ).
\end{equation*}
\end{proof}

\section{The minimum curvature-constrained Steiner point algorithm in 3D space}

We now present the minimum curvature-constrained Steiner point algorithm for determining the optimal location of the Steiner point in a 3-terminal network in 3D space. The optimal location of the Steiner point is obtained so as to minimise the total length of the network. An iterative process is introduced to first solve the projected problem in the horizontal plane before lifting the solution to 3D space. The location of the Steiner point is then determined in a plane in 3D space. This process is iterated till it converges to the optimal location. In what follows, we give a more detailed explanation of how to determine the location of the junction in the weighted planar version of the problem for both the odd and even topologies. We then explain how to lift the solution to 3D space. Finally, we explain how to update the weights when projecting the solution back onto a horizontal plane.

\subsection*{Solving the problem in the horizontal plane}
\subsubsection*{Determining the optimal location of the Steiner point for the odd topology}
We begin with the odd topology. Fig. \ref{FIG10} illustrates all the points that we consider when locating the Steiner point in the horizontal plane. 
\begin{figure}[h!]
	\centering
	\includegraphics[scale=0.7]{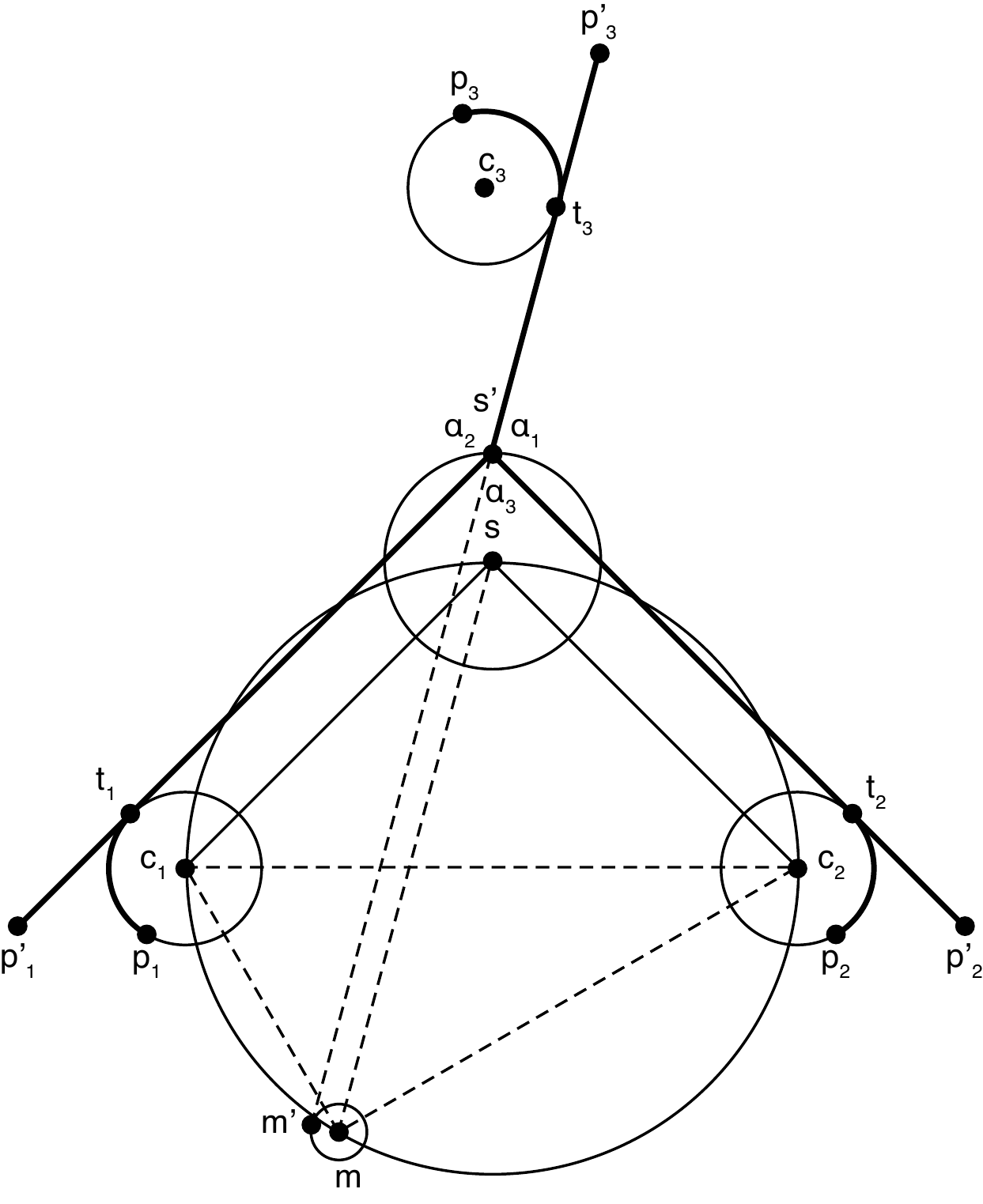}
	\caption{Optimal location of the Steiner point for the odd topology}
	\label{FIG10}
\end{figure}
\begin{enumerate}[1.]
	\item Find the three centres $c_1, c_2, c_3$ of the three chosen Dubins circles.\\
	\begin{figure}[h!]
		\centering
		\includegraphics[scale=0.6]{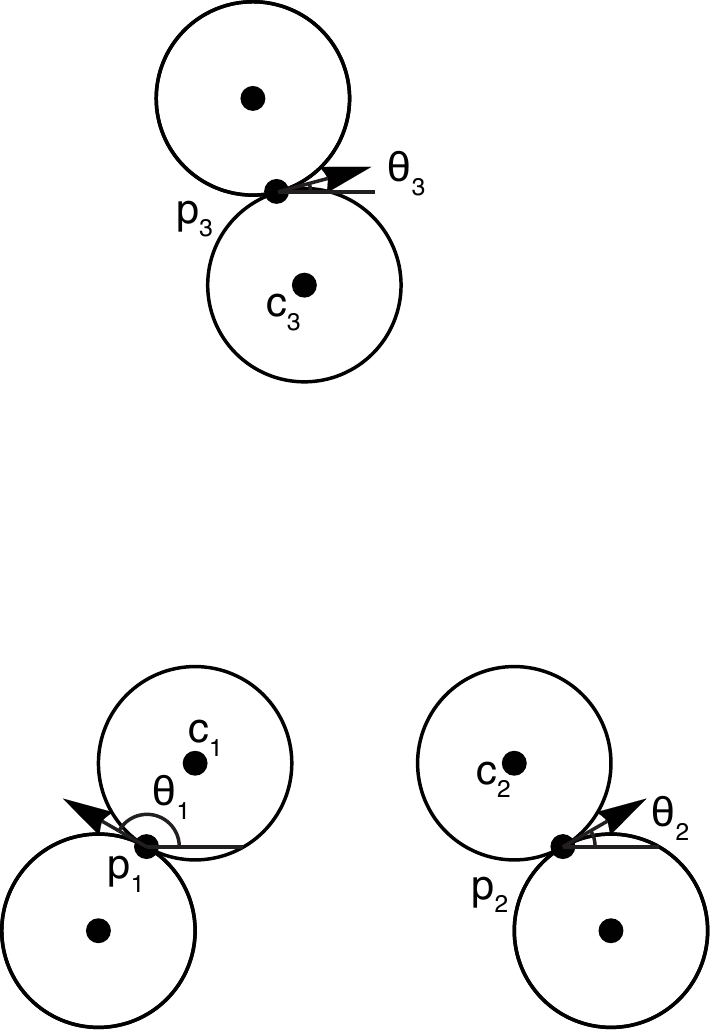}
		\caption{Possible Dubins circles incident to the given terminals}
		\label{FIG11}
	\end{figure}
These have coordinates $(c_{i}^x, c_{i}^y) = (x_i, y_i) \pm (-\sin \theta_{i}, \cos \theta_{i})$ with $i = 1, 2, 3$, as shown in Fig \ref{FIG11}. Note that there are two choices of Dubins circle for each directed point. So we need to keep track of eight possible networks and compare their lengths.
	
	\item Find the Melzak point $m=(m_x, m_y)$.\\
	Consider the centres $c_1$ and $c_2$ of two Dubins circles. The distance $d_3$ between the centres $c_1$ and $c_2$ is given by $d_3=\sqrt{(c_1^x-c_2^x)^2+(c_1^y-c_2^y)^2}$.  At the first iteration, we take the angles $\alpha_i$ with $i = 1, 2, 3$ at the junction to be equal to $2\pi/3$, in which case the weights $w_i$ are equal and $\triangle c_1c_2m$ equilateral. However, in subsequent iterations, the angles $\alpha_i$ with $i =1, 2, 3$ will not in general be equal and will need to be determined from the weights $w_i$. We will start with the special case in which the weights $w_i$ are equal and return to the issue of computing more general weights later. Since $\triangle c_1c_2m$ is an equilateral triangle,
	\begin{equation}
	(c_{i}^x-m_x)^2+(c_{i}^y-m_y)^2=d_3^2 \label{eq:c1}
	\end{equation}
for $i = 1, 2$. By solving Equation \ref{eq:c1} for $i = 1, 2$, we will get two possible solutions for $m$. We compare the distances from each of the two possible solutions to $p_3$. The point that gives the maximum distance is picked as the location for $m$.
	
	\item Find the tangent point $t_3$ on circle $C_3$.\\
	 The tangent point $t_3=(t_3^x,t_3^y)$ is the point where the Simpson line tangent to the Melzak circle with centre $m$ and radius $\sin (\alpha_1) \cot \left(\frac{\alpha_3}{2}\right) + \cos (\alpha_1)$ is tangent to $C_3$. At the first iteration, we take the angles $\alpha_i$ with $i = 1, 2, 3$ at the junction to be equal to $2\pi/3$, in which case the weights $w_i$ are equal and the Melzak circle degenerates to a point.
	
Now consider two vectors $\mathbf u_3$ and $\mathbf r_3$ along the Simpson line and radius of $C_3$, respectively, where $\mathbf u_3=(t_3^x-m'_{x})\mathbf i+(t_3^y-m'_{y})\mathbf j$ and $\mathbf r_3=(t_3^x-c_3^x)\mathbf i+(t_3^y-c_3^y)\mathbf j$. We obtain the equations
	\begin{eqnarray}
	(c_3^x-t_3^x)^2+(c_3^y-t_3^y)^2=1 \label{eq:c2}\\
	(t_3^x-m'_{x})(t_3^x-c_3^x)+(t_3^y-m'_{y})(t_3^y-c_3^y)=0. \label{eq:c3}
	\end{eqnarray}
	Equation \ref{eq:c2} is the equation of the circle $C_3$. The vectors $\mathbf u_3$ and $\mathbf r_3$ are perpendicular, so $\mathbf u_3. \mathbf r_3=0$, which is expressed in Equation \ref{eq:c3}.
	
	 Similarly, consider two vectors $\mathbf u_3$ and $\mathbf r_m$ along the Simpson line and radius of the Melzak circle $C_m$, respectively, where $\mathbf u_3=(t_3^x-m'_{x})\mathbf i+(t_3^y-m'_{y})\mathbf j$ and $\mathbf r_m=(m'_x-m_x)\mathbf i+(m'_y-m_y)\mathbf j$. We obtain the equations
	\begin{eqnarray}
	(m_x-m'_x)^2+(m'_y-m_y)^2=(\sin (\alpha_1) \cot \left(\frac{\alpha_3}{2}\right) + \cos (\alpha_1))^2 \label{eq:c4}\\
	(t_3^x-m'_{x})(m'_x-m_x)+(t_3^y-m'_{y})(m'_y-m_y)=0. \label{eq:c5}
	\end{eqnarray}
	Equation \ref{eq:c4} is the equation of the circle $C_m$. The vectors $\mathbf u_3$ and $\mathbf r_m$ are perpendicular, so  $\mathbf u_3. \mathbf r_m=0$, which is expressed in Equation \ref{eq:c5}. The coordinates of $m'$ and $t_3$ can be found by  solving simultaneous Equations \ref{eq:c2}, \ref{eq:c3}, \ref{eq:c4} and \ref{eq:c5}.
	\item Find the point $s=(s_{x},s_{y})$.\\
	First calculate the equation of the circle $C$ through the points $c_1,c_2,m$, in the form of $(x-a)^2+(y-b)^2=r^2$. We have
	\begin{eqnarray}
	(c_1^x-a)^2+(c_1^y-b)^2=r^2 \label{eq:c6}\\
	(c_2^x-a)^2+(c_2^y-b)^2=r^2 \label{eq:c7}\\
	(m_x-a)^2+(m_y-b)^2=r^2. \label{eq:c8}
	\end{eqnarray}
	By solving Equations \ref{eq:c6}, \ref{eq:c7}, \ref{eq:c8}, the coefficients $a,b,r$ can be determined. In addition, we have
	\begin{eqnarray}
	(s_x-a)^2+(s_y-b)^2=r^2 \label{eq:c9}\\
	\frac{s_y-m_y}{s_x-m_x} =\frac{t_3^y-m'_y}{t_3^x-m'_x} \label{eq:c10}.
	\end{eqnarray}
	 Equations \ref{eq:c9} and \ref{eq:c10} are obtained from the fact that the point $s$ lies on the circle $C$ and the line $ms$ is parallel to the Simpson line $m't_3$.
	The coordinates of $s=(s_{x},s_{y})$ can be calculated by solving Equations \ref{eq:c9} and \ref{eq:c10}.
	\item Find the junction $s'=(s'_x,s'_y)$.\\
	 The junction $s'$ lies on the intersection of the Simpson line $m't_3$ and the circle with centre $s$ and radius $\csc \left(\frac{\alpha_3}{2}\right)$. We have the equations
	\begin{eqnarray}
	 (s_x-s'_x)^2+(s_y-s'_y)^2= \csc^2(\alpha_3) \label{eq:c11}\\
	\frac{s'_y-m'_y}{s'_x-m'_x} =\frac{t_3^y-m'_y}{t_3^x-m'_x}. \label{eq:c12}
	\end{eqnarray}
	By solving Equations \ref{eq:c11} and \ref{eq:c12}, the coordinates of $s'$ can be found.
	
	\item Find the points of tangency $t_1,t_2$.\\
	The point $t_1=(t_1^x,t_1^y)$ is the tangent point of the line $s't_1$ and Dubins circle $C_1$. Consider two vectors $\mathbf u_1$ and $\mathbf r_1$ along the line $s't_1$ and radius of $C_1$, respectively, where $\mathbf u_1=(t_1^x-s'_x)\mathbf i+(t_1^y-s'_y)\mathbf j$ and $\mathbf r_1=(t_1^x-c_1^x)\mathbf i+(t_1^y-c_1^y)\mathbf j$.  We obtain the equations
	\begin{eqnarray}
	(c_1^x-t_1^x)^2+(c_1^y-t_1^y)^2=1 \label{eq:c13}\\
	(t_1^x-s'_x)(t_1^x-c_1^x)+(t_1^y-s'_y)(t_1^y-c_1^y)=0. \label{eq:c14}
	\end{eqnarray}
	Equation \ref{eq:c12} is the equation of the circle $C_1$. The vectors $\mathbf u_1$ and $\mathbf r_1$ are perpendicular, so $\mathbf u_1. \mathbf r_1=0$, which is expressed in Equation \ref{eq:c14}. The coordinates of $t_1$ can be found by solving Equations \ref{eq:c13} and \ref{eq:c14}.
	
	Similarly, the point $t_2=(t_2^x,t_2^y)$ is the tangent point of the line $s't_2$ and  Dubins circle $C_2$. Consider two vectors $\mathbf u_2$ and $\mathbf r_2$ along the line $s't_2$ and radius of $C_2$, respectively, where  $\mathbf u_2=(t_2^x-s'_x)\mathbf i+(t_2^y-s'_y)\mathbf j$ and $\mathbf r_2=(t_2^x-c_2^x)\mathbf i+(t_2^y-c_2^y)\mathbf j$. We obtain the equations
	\begin{eqnarray}
	(c_2^x-t_2^x)^2+(c_2^y-t_2^y)^2=1 \label{eq:c15}\\
	(t_2^x-s'_x)(t_2^x-c_2^x)+(t_2^y-s'_y)(t_2^y-c_2^y)=0. \label{eq:c16}
	\end{eqnarray}
	Equation \ref{eq:c15} is obtained from the radius of $C_2$. The vectors $\mathbf u_2$ and $\mathbf r_2$ are perpendicular, so $\mathbf u_2. \mathbf r_2=0$, which is expressed in Equation \ref{eq:c16}. The coordinates of $t_2$ can be found by  solving Equations \ref{eq:c15} and \ref{eq:c16}.
	
	\end{enumerate}
	
	\subsubsection*{Determining the optimal location of the Steiner point for the even topology}
	The method is similar to that used for the odd topology, and we abbreviate it somewhat.
\begin{figure}[h!]
	\centering
	\includegraphics[scale=0.8]{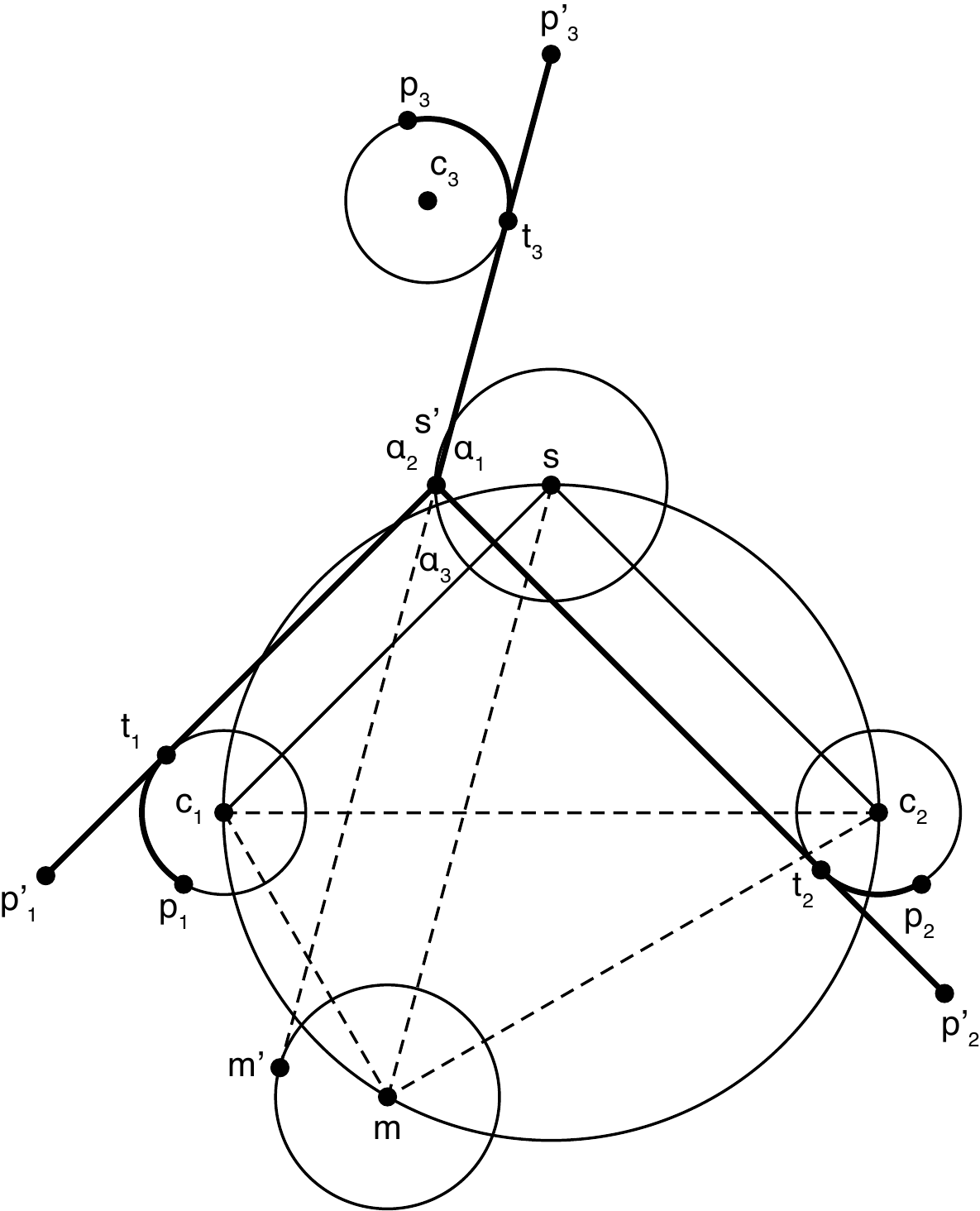}
	\caption{Optimal location of the Steiner point for the even configuration}
	\label{FIG12}
\end{figure}
\begin{enumerate}[1.]
	\item Find the points of tangency $m'$ and $t_3$ on the Melzak circle and circle $C_3$, respectively, that determine the Simpson line.\\
	Consider the Melzak circle $C_m$ with centre $m$ and radius $\sin (\alpha_1) \tan \left(\frac{\alpha_3}{2}\right) - \cos (\alpha_1)$, as shown in Fig. \ref{FIG12}.
	The point $m'$ is such that a line through $m'$  is tangent to both the circles $C_m$ and $C_3$.  The point $t_3=(t_3^x,t_3^y)$ is the tangent point of the Simpson line through $m'$ and the Dubins circle $C_3$.\\
	Consider two vectors $\mathbf u_3$ and $\mathbf r_3$ along the Simpson line and radius of $C_3$, respectively, where $\mathbf u_3=(t_3^x-m'_{x})\mathbf i+(t_3^y-m'_{y})\mathbf j$ and $\mathbf r_3=(t_3^x-c_3^x)\mathbf i+(t_3^y-c_3^y)\mathbf j$. We obtain the equations
	\begin{eqnarray}
	(c_3^x-t_3^x)^2+(c_3^y-t_3^y)^2=1 \label{eq:c35}\\
	(t_3^x-m'_{x})(t_3^x-c_3^x)+(t_3^y-m'_{y})(t_3^y-c_3^y)=0. \label{eq:c36}
	\end{eqnarray}
	Equation \ref{eq:c35} is obtained from the radius of $C_3$. The vectors $\mathbf u_3$ and $\mathbf r_3$ are perpendicular, so $\mathbf u_3. \mathbf r_3=0$, which is expressed in Equation \ref{eq:c36}. 
	
	Similarly, consider two vectors $\mathbf u_3$ and $\mathbf r_m$ along the Simpson line and radius of $C_m$, respectively, where $\mathbf u_3=(t_3^x-m'_{x})\mathbf i+(t_3^y-m'_{y})\mathbf j$ and $\mathbf r_m=(m'_x-m_x)\mathbf i+(m'_y-m_y)\mathbf j$. We obtain the equations
	\begin{eqnarray}
	 (m_x-m'_x)^2+(m'_y-m_y)^2= (\sin (\alpha_1) \tan \left(\frac{\alpha_3}{2}\right) - \cos (\alpha_1))^2 \label{eq:c37}\\
	(t_3^x-m'_{x})(m'_x-m_x)+(t_3^y-m'_{y})(m'_y-m_y)=0. \label{eq:c38}
	\end{eqnarray}
	Equation \ref{eq:c37} is the equation of the circle $C_m$. The vectors $\mathbf u_3$ and $\mathbf r_m$ are perpendicular, so  $\mathbf u_3. \mathbf r_m=0$, which is expressed in Equation \ref{eq:c38}. The coordinates of $m'$ and $t_3$ can be found by  solving simultaneous Equations \ref{eq:c35}, \ref{eq:c36}, \ref{eq:c37} and \ref{eq:c38}.
	
	 \item Find the junction $s'=(s'_{x},s'_{y})$.\\
	The point $s'$ lies on the intersection of the Simpson line $m't_3$ and the circle with centre $s$ and radius $\sec \left(\frac{\alpha_3}{2}\right)$. We obtain the equations
	\begin{eqnarray}
	(s_x-s'_x)^2+(s_y-s'_y)^2= \sec^2(\alpha_3) \label{eq:c39}\\
	\frac{s'_y-m'_y}{s'_x-m'_x} =\frac{t_3^y-m'_y}{t_3^x-m'_x}. \label{eq:c40}
	\end{eqnarray}
	The coordinates of $s'$ can be found by  solving  Equations \ref{eq:c39} and \ref{eq:c40}.
\end{enumerate}

\subsection*{Solving the problem in 3D space}
\begin{enumerate}[1.]
	\item Find the lifted points $p_{1}'=(x'_{1},y'_{1})$, $p_{2}'=(x'_{2},y'_{2})$, $p_3'=(x'_3,y'_3)$.\\
	The coordinates of terminals $p_{1}=(x_{1},y_{1},z_{1}),p_{2}=(x_{2},y_{2},z_{2}),p_3=(x_3,y_3,z_3)$ are given. 
	\begin{figure}[h!]
		\centering
		\includegraphics[scale=0.8]{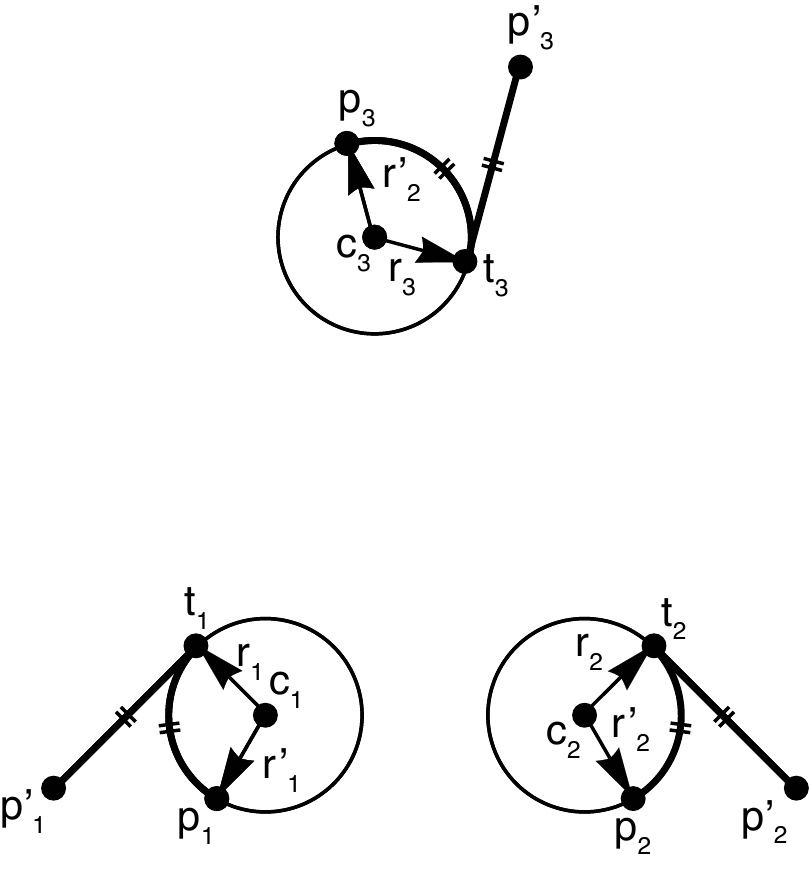}
		\caption{Locations of the points $p_1', p_2', p_3'$}
		\label{FIG13}
	\end{figure}
	The point $p'_i$ with $i = 1, 2, 3$ is such that $|t_ip_i|=|t_ip'_i|$, where $|t_ip_i|=\beta_i$. We have
	\begin{equation}
	\mathbf r'_i. \mathbf r_i= ||\mathbf r'_i|| \hspace{1mm} ||\mathbf r_i||\cos \beta_i,
	\end{equation}
from which it follows that
	\begin{equation}
	\beta_i =\arccos \bigg(\frac{\mathbf r'_i. \mathbf r_i}{||\mathbf r'_i|| \hspace{1mm} ||\mathbf r_i||}\bigg)
	\end{equation}
for $i = 1, 2, 3$, where $\mathbf r_i=(t_i^x-c_i^x)\mathbf i+(t_i^y-c_i^y)\mathbf j$ and $\mathbf r'_i=(x_i-c_i^x)\mathbf i+(y_i-c_i^y)\mathbf j$. In addition, we have
	\begin{eqnarray}
	(x'_i-t_i^x)^2+(y'_i-t_i^y)^2=|t_ip'_i|^2=\beta_i^2 \label{eq:c17}\\
	\frac{s'_y-y'_i}{s'_x-x'_i} =\frac{s'_y-t_i^x}{s'_x-t_i^y} \label{eq:c18}
	\end{eqnarray}
for $i = 1, 2, 3$. By solving Equations \ref{eq:c17} and \ref{eq:c18}, the coordinates of $p'_i$ with $i = 1, 2, 3$ can be found.\\
	\item Find the plane on which the Steiner point lies.\\
	The equation of a plane is given by $ax+by+cz=d$, where $d=ax_{0}+by_{0}+cz_{0}$ for any point $(x_{0},y_{0},z_{0})$ in the plane. The normal vector $(a, b, c)$ is determined by the cross product of two vectors, $p'_0p'_1 = (x'_1 - x'_0)\hat{i}+(y'_1 - y'_0)\hat{j} + (z_1 - z_0)\hat{k}$ and $p'_0p'_2 = (x'_2 - x'_0)\hat{i} + (y'_2 - y'_0)\hat{j} + (z_2 - z_0)\hat{k}$. That is, $a=(y'_{1}-y'_{0})(z_{2}-z_{0})-(y'_{2}-y'_{0})(z_{1}-z_{0})$, $b=(x'_{2}-y'_{0})(z_{1}-z_{0})-(x'_{1}-x'_{0})(z_{2}-z_{0})$, $c=(x'_{1}-x'_{0})(y'_{2}-y'_{0})-(x'_{2}-x'_{0})(y'_{1}-y'_{0})$. The Steiner point lies on the intersection of this plane and the vertical line $(s'_x,s'_y)$, hence
	\begin{equation}
	a(s'_x-x_{0})+b(s'_y-y_{0})+c(s'_z-z_{0})=0,
	\end{equation}
from which it follows that
	\begin{equation}
	s'_z=z_{0}- \frac{a(s'_x-x_{0})+b(s'_y-y_{0})}{c}.  \label{eq:c24}
	\end{equation}
	By solving Equation \ref{eq:c24}, the $z$ coordinate of the Steiner point can be obtained.
	\item Find the gradients of the line segments  $p'_1s',p'_2s',p'_3s'$.\\
	Three gradients $g_1,g_2,g_3$ \cite{Brazil_a} are defined for three straight line segments
	\newline$p'_1s',p'_2s',p'_3s'$ and given by Equation \ref{eq:c25} for $i = 1, 2, 3$ respectively. Thus, for $i 
	= 1, 2, 3$
	\begin{equation}
	g_i= \frac{|s'_z-z_i|} {\sqrt{(s'_x-x'_i)^2+(s'_y-y'_i)^2}}. \label{eq:c25}
	\end{equation}
	\item Find the weights of edges $p'_1s',p'_2s',p'_3s'$.\\
	The weights $w_1,w_2,w_3$ are used to project the solution back onto the horizontal plane. The gradient of a line segment can be used to obtain the weight of the corresponding edge. The weights for line segments $p'_1s',p'_2s',p'_3s'$ are $w_1,w_2,w_3$ respectively, where $w_i=1/\sqrt{1+g_i^2}$ for $i = 1, 2, 3$. 
	\item Find the angles $\alpha_1, \alpha_2, \alpha_3$.\\
	The optimum angles between the three edges can be computed as functions of the three weights, that is,
	\begin{equation}
	\alpha_i = \pi - \arccos\bigg(w_i\frac{w_0^2 + w_1^2 + w_2^2 - 2w_i^2}{2w_0w_1w_2}\bigg) \label{eq:c28}
	\end{equation}
	\begin{figure}[h!]
		\centering
		\includegraphics[scale=0.6]{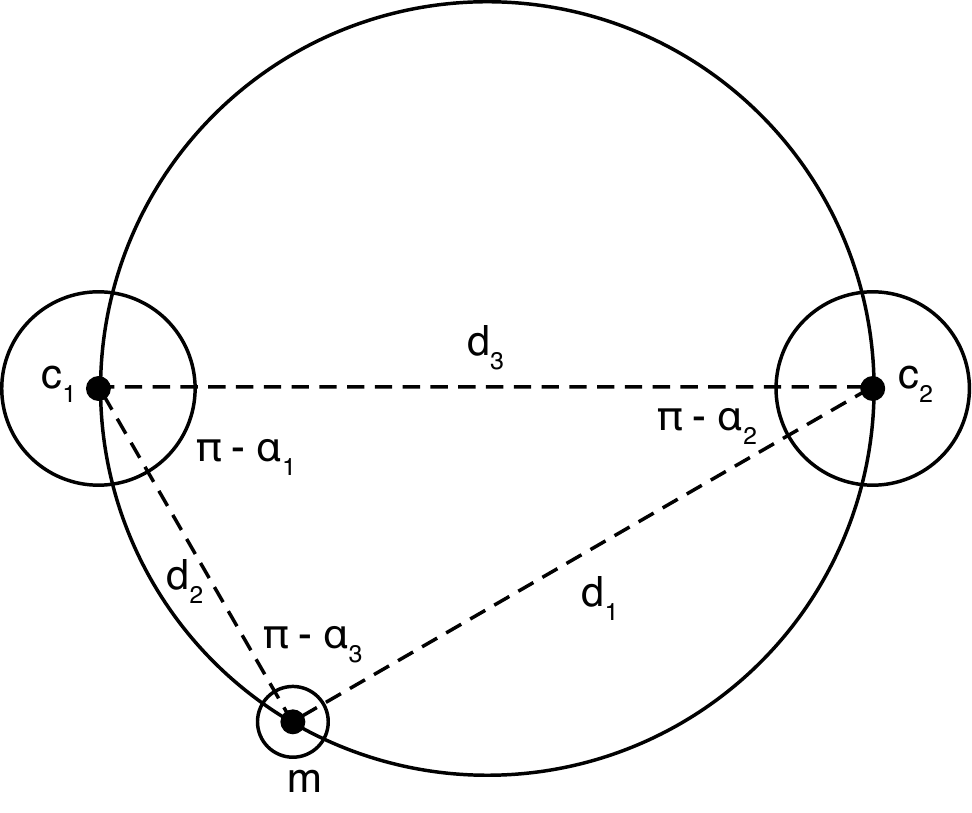}
		\caption{Locating the point $m$}
		\label{FIG14}
	\end{figure}\\
	for $i = 1, 2, 3$. Now the triangle $c_1c_2m$ is drawn based on the angles calculated from Equation \ref{eq:c28}. This is shown in Fig. \ref{FIG14}. The coordinates of $m$ can be calculated as follows.
	By applying the sine rule to the triangle $c_1c_2m$:
	\begin{eqnarray}
	 d_i=\frac{d_3\sin(\pi - \alpha_i)}{\sin(\pi - \alpha_3)} \label{eq:c31}
	\end{eqnarray}
	for $i = 1, 2$. The distances from $c_1$ to $m$ and $c_2$ to $m$ are $d_2$ and $d_1$, respectively, where
	\begin{eqnarray}
	(c_1^x-m_x)^2+(c_1^y-m_y)^2=d_2^2 \label{eq:c33}\\
	(c_2^x-m_x)^2+(c_2^y-m_y)^2=d_1^2. \label{eq:c34}
	\end{eqnarray}
	By solving Equations \ref{eq:c33} and \ref{eq:c34}, the coordinates of $m$ can be found.
\end{enumerate}

The steps discussed in the previous section can be expressed in an algorithm.
\begin{algorithm}[h!]
	\DontPrintSemicolon
	\KwIn{The directed points $p_1$, $p_2$, $p_3$, the radius of the curvature $r$ and tolerance $\epsilon$}
	\KwOut{The optimal location of the Steiner point and optimal length of the network}
	Calculate the centres $c_1$, $c_2$, $c_3$ of three Dubins circles. \;
	{ Initialisation: \textit{$\alpha_1(0)=2\pi/3$, $\alpha_2(0)=2\pi/3$, $\alpha_3(0)=2\pi/3$}\; $i=0$\; \Repeat{ $|length(i)-length(i-1)| < \epsilon$}{%
			\textbf{In the horizontal plane}\\
			Find the Melzak point $m$ using two different methods depending on the topology, the tangent points $t_3,t_2,t_1$, the junction $s'$, and lifted the points $p'_1$, $p'_2$, $p'_3$.\\
			\textbf{In 3D space}\\
			Find the plane on which $p_1',p_2',p_3'$ lie, the $z$ coordinate of the Steiner point, the gradients of the line segments  $p'_1s',p'_2s',p'_3s'$, the weights of edges $p'_1s',p'_2s',p'_3s'$.\\
			Find the angles $\alpha_1, \alpha_2, \alpha_3$ based on the weights.
			\\ \nl $i = i + 1$\; 
		}
	}
	\caption{The minimum curvature-constrained Steiner point algorithm}
	\label{algo:2}
\end{algorithm}

	\begin{figure}
			\centering
			\includegraphics[width=\textwidth]{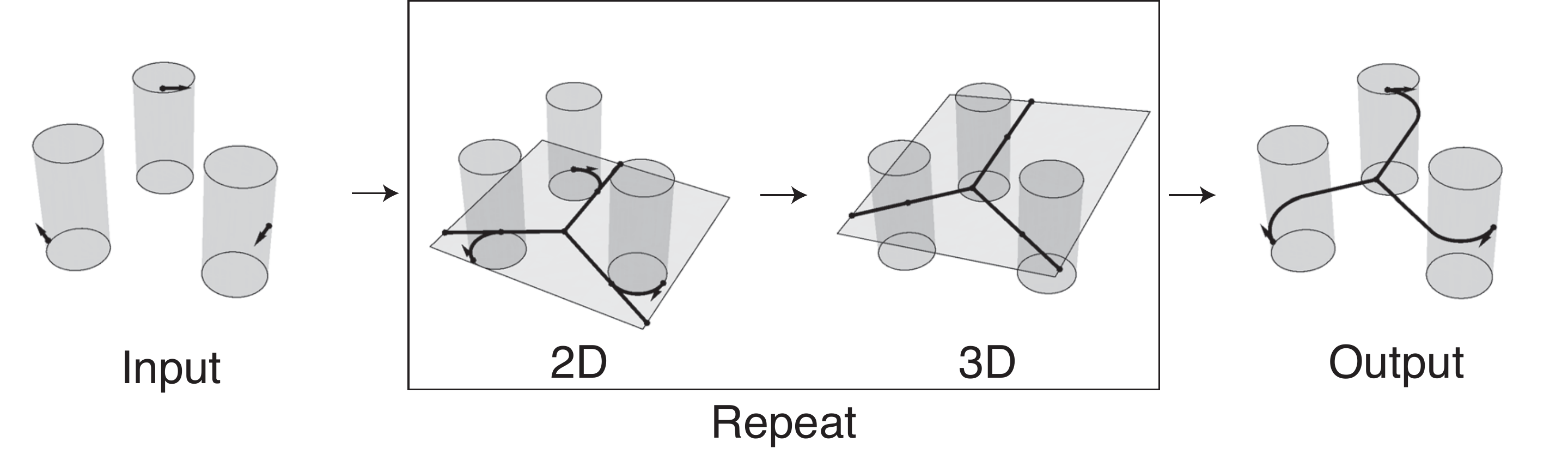}
			\caption{The minimum curvature-constrained Steiner point algorithm}
			\label{FIG15}
		\end{figure}

\newpage \section{Convergence of the algorithm}\label{sec:2}
The admissible network converges to the minimum Dubins network with each iteration of the algorithm, because the length of the admissible network is convex with respect to the position of the Steiner point. Here, we show just that.

\begin{proposition}
The length of an admissible network is convex with respect to the position of the Steiner point.
\end{proposition}
	
	\begin{proof}
Let $p_1 p_2 p_3$ be a curve of type $CS$ where:
		
		\begin{itemize}
			\item $p_1 p_2$ is a helical arc of radius 1.
			\item $p_2 p_3$ is a straight segment.
		\end{itemize}
		
		\begin{figure}
			\centering
			\includegraphics[width=0.45\textwidth]{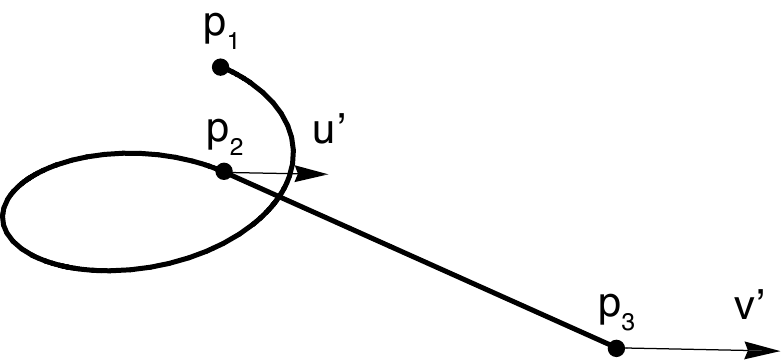}
			\caption{An edge of the admissible network}
			\label{FIG15}
		\end{figure}

In Figure \ref{FIG15}, $p_1$ is held fixed while the point $p_2$ moves along the vector $\mathbf{u}'$ and the point $p_3$ along the vector $\mathbf{v}'$. Let $L$ denote the length of $p_1 p_2 p_3$. Note that the helical arc is just a straight line segment wrapped around a cylinder. In other words, we can represent the helical arc by a vector $\vv{p_1 p_2}$ lying in the plane. Then according to Proposition \ref{PROP1}, the first variation of length of the helical arc $p_1 p_2$ is given by $\mathbf{u}' \cdot \hat{p_1 p_2}$, where the perturbation of $p_2$ in the direction of $\mathbf{u}'$ is tangent to the cylinder. On the other hand, the first variation of length of the straight line segment $p_2 p_3$ is given by $(\mathbf v' - \mathbf u') \cdot \hat{p_2 p_3}$. Summing these two terms, we obtain for the first variation of length of $p_1 p_2 p_3$
		
		\begin{equation*}
		L' = \mathbf{u}' \cdot \hat{p_1 p_2} + (\mathbf v' - \mathbf u') \cdot \hat{p_2 p_3} = \mathbf v' \cdot \hat{p_2 p_3}.
		\end{equation*}
		
		By the product rule, the second variation of length of $p_1 p_2 p_3$ is
		
		\begin{equation*}
		L'' = \mathbf{v}'' \cdot \hat{p_2 p_3} + \mathbf{v}' \cdot \hat{p_2 p_3}'.
		\end{equation*}
		
		Note that the first term is zero, because the perturbation of $p_3$ in the direction of $\mathbf{v}'$ is supposed to be linear. It is geometrically obvious that the angle between the vectors $\mathbf{v}'$ and $\hat{p_2 p_3}'$ is no greater than $\pi/2$. Hence the second variation of $L$ is nonnegative and $L$ is convex. Since the weighted sum of convex functions is also convex, it follows that the length of the network is convex with respect to the position of the Steiner point. \qed
	\end{proof}
	
	We note that in the 3-terminal case, a weighted Dubins network (weighted Steiner tree) is the projection of an unweighted Dubins network (an unweighted Steiner tree) in $\mathbb{R}^3$. Moreover, the 2D gradient is the projected 3D gradient. Hence, with regards to the algorithm, the projected admissible network converges if and only if the admissible network converges.

\section{Conclusion}\label{sec:7} 

We have given algorithms to find minimal Dubins networks for three directed points in the plane and in 3D space. For applications to mining, it is necessary to also incorporate a gradient constraint so the latter algorithm is only useful if the resulting network has gradient less than the maximum allowed. In \cite{Brazil_a}, there is a discussion of how to find gradient constrained Steiner trees in 3D space. So a challenge is to merge the methods in this paper and \cite{Brazil_a} to find a full solution for the 3D gradient constrained Dubins problem. 

For more directed points in the plane, our 3 directed point  method will generalise to find Dubins networks, so long as the points are sufficiently far apart relative to the turning circle constraints. If the points are relatively close together, solutions which look like Steiner trees with arcs of circles near the directed points will not be possible. 


\end{document}